\documentclass{amsart}
\usepackage{amssymb,mathrsfs}%,skak}
\usepackage{amsmath, amsthm,amsfonts,amssymb,url,color,epsfig,graphics}
\usepackage[english]{babel}
\usepackage{float}
\usepackage{multirow}
\usepackage{hyperref}
\hypersetup{pdfborder=0 0 0}

\usepackage{color}

\usepackage{enumitem}

\usepackage{graphics}
\usepackage{comment}
\usepackage{graphicx}

\usepackage{geometry}
\newtheorem{proposition}{Proposition}[section]
\newtheorem{theorem}[proposition]{Theorem}

\newtheorem{lemma}[proposition]{Lemma}
\newtheorem*{bootstrap*}{Bootstrap Step}
\theoremstyle{definition}
\newtheorem{definition}[proposition]{Definition}
\newtheorem{remark}[proposition]{Remark}

\numberwithin{equation}{section}
\newtheorem{assumption}[proposition]{Assumption}

\definecolor{darklavender}{rgb}{0.56, 0.0, 1.0}
\definecolor{green}{rgb}{0.0, 0.5, 0.0}
\definecolor{blue}{rgb}{0.0, 0.0, 0.55}

\newcommand\eps{\varepsilon}

\def\Re{{\rm Re}}

\newcommand\de{{\partial}}

\newcommand\RR {{\mathbb R}}
\newcommand{\CC}{\mathbb{C}}
\newcommand{\QQ}{\mathbb{Q}}
\def\L{\lambda} 
\def\R{\mathbb{R}} 
\def\k{\textbf{k}} 
\def\W{\mathcal{W}} 
\def\Wbl{\mathcal{W}_{\rm{BL}}}

\def\G{\chi}

\newcommand{\bcb}[1]{#1}
\newcommand{\be}{\begin{equation}}
\newcommand{\ee}{\end{equation}}
\newcommand{{\kk}}{|\k|}

\newcommand{{\kko}}{|\k^0|}

\usepackage{etoolbox}
\patchcmd{\subsubsection}{\itshape}{\itshape\bfseries}{}{} % THIS CHANGE ITALIC FONT OF SUBSUBSECTION INTO BOLD-ITALIC

\title[]{Linear boundary layer analysis \\of the near-critical reflection of internal gravity waves \\with different sizes of viscosity and diffusivity}
\author[R.\ Bianchini]{Roberta Bianchini}
\address{Consiglio Nazionale delle Ricerche, Istituto per le Applicazioni del Calcolo, 00185 Rome, Italy}
\email{roberta.bianchini@cnr.it}
\author[G.\ Orr\`u]{Gianluca Orr\`u}
\address{Sapienza University of Rome, Department of Mathematics}
\email{orru.1602245@studenti.uniroma1.it}

\begin{document}
\maketitle

\begin{abstract} 
The aim of this work is to make a further step towards the understanding of the near-critical reflection of internal gravity waves from a slope in the more general and realistic context where the size of viscosity $\nu$ and the size of diffusivity $\kappa$ are different. 
In particular, we provide a systematic characterization of boundary layers (boundary layer wave packets) decays and sizes depending on the order of magnitude of viscosity and diffusivity.
We can construct an $L^2$ stable approximate solution to the linear near-critical reflection problem under the scaling assumption of Dauxois \& Young JFM 1999, where either viscosity of diffusivity satisfies a precise scaling law in terms of the criticality parameter.
\end{abstract}

\tableofcontents

\section{Introduction}
This work provides new results on the analysis of the \emph{linear} near-critical reflection problem for the two-dimensional Boussinesq system 
\begin{align}
\de_t u - b \sin \gamma + \de_x p &= \nu  \Delta u,\notag\\
\de_t w - b \cos \gamma + \de_y p &=  \nu \Delta w,\notag\\
\de_t b + N^2(u\sin \gamma  + w \cos \gamma) & = \kappa  \Delta b,\notag\\
\de_x u + \de_y w&=0,
\label{eq:system}
\end{align}
in the half space $\mathbb{R}^2_+$, with viscosity $\nu>0$ and diffusivity $\kappa>0$, while $N^2$ is a positive constant which plays the role of the \emph{buoyancy frequency} (further details later on). The system is endowed with \emph{no-slip} boundary conditions for the velocity field and \emph{no-flux} boundary condition for the buoyancy term, i.e.
\begin{align}
u_{|y=0}=w_{|y=0}=\de_y b_{|y=0}=0.\label{eq:cond-BL}
\end{align}
The linear inviscid approximation of the Boussinesq system (i.e. \eqref{eq:system} with $\nu=\kappa=0$) supports the propagation of waves, which are called \emph{internal gravity waves}. This can be seen by analo\-gy with geometric optics for acoustic waves (see for instance \cite{Metivier}): taking the Fourier transform of \eqref{eq:system} with $\nu=\kappa=0$, one can define the solution $(u, w, b) \in \R^3$ as a sum of plane waves $(\widehat u, 
\widehat w,
\widehat b
)^T e^{-i\omega_{k,m} t + i kx+imy}$, where $(k,m) \in \R^2$ are the Fourier variables, whose \emph{dispersion relation} and Fourier coefficients are respectively given by 
\begin{align}\label{eq:disprel}
\omega_{k,m}^2 = N^2\frac{(k \cos \gamma-m\sin \gamma)^2}{k^2+m^2}=(N \sin \theta)^2, \qquad 
X_{k,m}:=\begin{pmatrix}\widehat u\\
\widehat w\\
\widehat b\\
\end{pmatrix} = \begin{pmatrix}
1\\
-\frac km\\
N\frac{i(k \cos \gamma - m \sin \gamma)}{m \omega_{k,m}}
\end{pmatrix},
\end{align}
with $(\sqrt{k^2+m^2}, \frac \pi 2-(\theta+\gamma))$ being the polar coordinates of the wavenumber $(k,m)$ in Fourier space, in the reference system $(x,y)$.

The physical phenomenon which is analyzed here is the reflection of those internal waves from a sloping flat boundary of an arbitrary but fixed angle $\gamma$. As widely discussed in \cite{BDSR19}, since the notion of \emph{propagation} does not make sense for a single plane wave, we work with \emph{wave packets}. Our typical wave packet is a simple linear superposition of plane waves, as defined in \cite{BDSR19}.
More precisely, the definition of the wave packet hitting the boundary $y=0$ (incident wave packet) reads as follows
\begin{align}\label{def:inc}
\W_{\rm{inc}}(x,y):= \int_{\R^2} \widehat{A}(k,m) X_{k,m} e^{-i\omega_{k,m} t+ikx+imy} \, dk \, dm,
\end{align}
where
\begin{enumerate}
\item the eigenvector $X_{k,m} \in \R^3$ is given by \eqref{eq:disprel}, where the pressure can be recovered for instance using the first equation of \eqref{eq:system} (with $\nu=0$), which in Fourier reads
$$-i\omega_{k,m}\widehat u - \sin \gamma \widehat b + i k \widehat p =0 \quad \leftrightarrow \quad \widehat{p}=\frac 1k [ \omega_{k,m}+\sin \gamma \frac{(k\cos \gamma-m\sin \gamma)}{m\omega_{k,m}}],$$
\item the time frequency 
\begin{align*}
\omega_{k,m} = \omega_{k,m}^\pm \text{\, if \,} \nabla_{k,m} \omega^\pm_{k,m} (k,m) \cdot (0,1)^{\rm{T}} < 0, \text{\, with \,}
\omega_{k,m}^\pm=\pm\frac{k \cos \gamma-m \sin \gamma}{\sqrt{k^2+m^2}}, 
\end{align*}
\item the function 
\begin{align}\label{def:psi}
\widehat{A}(k,m):= \frac{1}{\eps^2}\sum_{\pm} \G (\eps^{-2}(k\pm k_0))  \G (\eps^{-2}(m\pm m_0)),
\end{align}
for any $C^\infty$ compactly supported function $\G$. 
\end{enumerate}
The choice of item (2) in the above list is motivated by the fact that we consider the propagation of the incident wave in \eqref{def:inc} in the upper half plane $y \ge 0$. Then, the \emph{incident} wave hits the boundary provided that it propagates downwards. This is in fact the reason why we impose that $\nabla_{k,m} \omega^\pm_{k,m} (k,m) \cdot (0,1)^{\rm{T}} < 0$. 
The criticality of the near-critical reflection phenomenon is measured in terms of the difference between the angle of the incident internal wave ($\theta$ with $\omega_{k,m}^2=(N\sin \theta)^2)$ and the angle $\gamma$ of the slope. The reflection of internal waves is called \emph{near-critical} (see \cite{DY1999}) if the difference (hereafter called \emph{criticality parameter})
\begin{align}\label{def:critical-relation}
\zeta:=\omega_{k,m}^2 - \sin^2 \gamma
\end{align}
is a small parameter. We briefly explain the meaning of this critical setting.
First, system \eqref{eq:system} is the rotated version (of angle $\gamma$) of the linear non-rotated 2d Boussinesq system in the original cartesian coordinates $(x_1, x_2)$, which reads
\begin{align}\label{eq:Boussinesq-nonrot}
\de_t \tilde u + \de_{x_1} \tilde p & = \nu \Delta \tilde u,\notag\\
\de_t \tilde w - \tilde b + \de_{x_2} \tilde p &= \nu \Delta \tilde w,\notag \\
\de_t \tilde b + N^2 \tilde w &= \kappa \Delta \tilde b,\notag \\
\de_{x_1} \tilde u + \de_{x_2} \tilde{w}&=0,
\end{align} 
in the upper-slope region of angle $\gamma$ determined by $\de_\Omega= \{ (x_1, x_2) \, | \, x_2-(\tan \gamma) x_1=0\}$, i.e. the domain where system \eqref{eq:Boussinesq-nonrot} is considered is the upper region identified by a slope of inclination $\gamma$ with respect to the horizontal $x_2=0$. 
The related boundary conditions are given by $\tilde u|_{\de_\Omega}=\tilde w|_{\de_\Omega}=(\nabla \cdot \overrightarrow{n}) b|_{\de_\Omega}=0$, where $\overrightarrow{n}=(\sin \gamma, -\cos \gamma)$ is the unit vector normal to the boundary $\de_\Omega$, see Figure \ref{fig:slope-coordinates}. 
\begin{figure}[h!]\label{fig:slope-coordinates}
\includegraphics[scale=0.2]{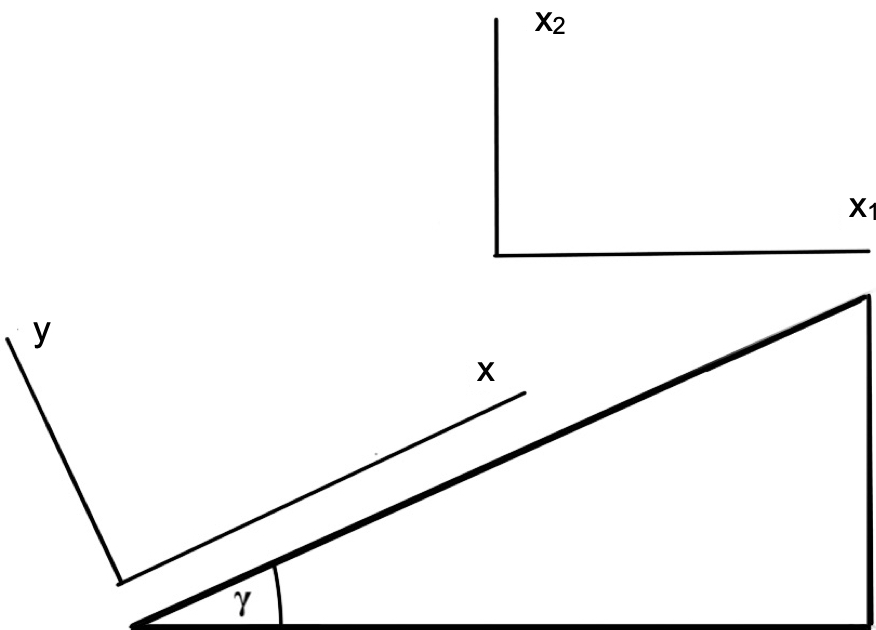}
\caption{Cartesian coordinates $(x_1, x_2)$ and rotated coordinates $(x,y)$.}
\end{figure}

Next, performing a plane wave (Fourier) analysis of the non-rotated system \eqref{eq:Boussinesq-nonrot} (as previously sketched for the rotated system \eqref{eq:system}), one obtains that the non-rotated dispersion relation is
\begin{align}\label{eq:omega-original}
\omega^2_{k_1, k_2} = N^2\frac{k_1^2}{k_1^2+k_2^2}=N^2 \mathcal{R}_1^2=N^2\sin^2\theta,
\end{align}
where $(k_1, k_2)$ are the non-rotated frequency variables in the original reference system $(x_1, x_2)$, $\mathcal{R}_1$ is the symbol of the first component of the Riesz transform, and $\theta$ is the angle between the frequency vector $(k_1, k_2)$ and the vertical axis $x_1=0$ (in other words, $(\kk, \frac \pi 2-\theta)$ are the polar coordinates of $(k_1, k_2)$). 
\begin{figure}[h!]
\includegraphics[scale=0.4]{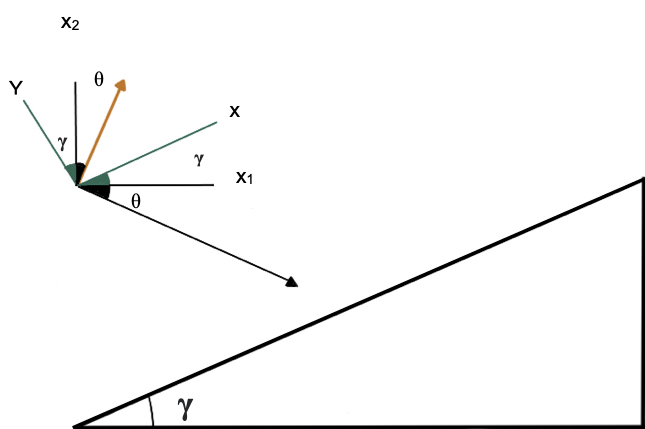}
\caption{Incident wave of frequency $\omega^2_{k,m}=(\sin \theta)^2$ hitting the slope of angle $\gamma$. The orange arrow is the wavenumber of components $(k_1, k_1)$ in $(x_1, x_2)$. The downward black arrow is the incident wave.}
\label{fig:slope2}
\end{figure}
%%%%%
This way,  the angle $\theta$ identifies the direction of the frequency vector (wavenumber) of components $(k_1, k_2)$ in the reference system $(x_1, x_2)$ (and of components $(k,m)$ in the rotated reference system $(x,y)$) of each superimposed plane wave $e^{-i\omega_{k,m} t +i kx + imy}$ inside the integral (the wavepacket) \eqref{def:inc}.
%%%%%

The reflection of plane waves is near-critical when $\sin^2\theta$ and $\sin^2\gamma$ are very close, where $\gamma$ is the slope inclination, see {Figure \ref{fig:slope2}}. The reason why this situation where $\zeta=\omega_{k,m}^2-\sin^2\gamma \sim 0$ is critical is due the anisotropic nature of the dispersion relation $\omega_{k,m}^2=N^2 \sin^2\theta$ of internal gravity waves. In two space dimensions, if $\omega_{k,m}^2 \in \R$ is fixed, then there are four frequency vectors $ (\pm k_1, \pm k_2)$ fulfilling the dispersion relation \eqref{eq:omega-original}. This implies that a given fixed frequency $\omega_{k,m}^2$ is responsible for the ''propagation'' of four internal waves in four specific directions, which are orthogonal to the four frequency vectors $ (\pm k_1, \pm k_2)$. When an incident internal gravity wave hits the slope, the anisotropic nature of this peculiar dispersion relation generates a reflected wave propagating in the vicinity of the slope if the inclination $\theta$ of the frequency vector and the angle of the slope $\gamma$ are very close. In particular, the reflected wave almost lying along the slope has ''infinite'' amplitude (i.e. the amplitude is a negative power of the small criticality parameter $\zeta$ in \eqref{def:critical-relation}). This phenomenon causes therefore an energy focusing along the boundary when the reflection is near-critical, i.e. $\zeta \sim 0$. In the presence of (small) viscosity and diffusivity as in \eqref{eq:Boussinesq-nonrot} with $\nu>0, \kappa>0$, the reflected wave is replaced by boundary layers (of high amplitude), see  \cite{DY1999, BDSR19}.
The physical literature on the near-critical reflection of internal waves is quite vast: the first investigation goes back to 1999 and it is due to Dauxois \& Young \cite{DY1999} (earlier physical experiments were realized by the group of Leo Mass, \cite{Maas}), while this is still an active direction of research in physics, see for instance \cite{kakaota}. In the Dauxois \& Young physical setting, see \cite{DY1999}, the considered scaling relation is the following. For an arbitrarily small parameter $\eps>0$, let us adimensionalize and set (either) the viscosity $\nu$ (or the diffusivity $\kappa$) as 
\begin{align}\label{eq:scaling-visc}
\nu=\nu_0 \eps^6 \qquad (\text{or} \quad \kappa=\kappa_0 \eps^6).
\end{align}
Then the criticality parameter in terms of the size of the viscosity reads
\begin{align}\label{eq:zeta}
\zeta \sim \eps^{2}.
\end{align}
An additional assumption is crucially adopted in \cite{DY1999} (and consequently in \cite{BDSR19}): both those studies rely on the scaling assumption $\nu \sim \kappa \sim \eps^6$, i.e. \emph{viscosity and diffusivity have the same size}, which is given in terms of the size of the criticality parameter $\zeta \sim \eps^2$. However, \emph{the order of magnitude of viscosity and diffusivity is very different in several physical scenarios}, for instance in the case of deep oceanic waters (see \cite{Dauxois-new}). 
The aim of this work is to make a further step towards the understanding of the near-critical reflection of internal gravity waves from a slope in the more general and realistic context where the size of viscosity $\nu$ and the size of diffusivity $\kappa$ are different. In particular, $\nu>0$ and $\kappa>0$ will be both small, according to the physical predictions \cite{DY1999}, but their order of smallness (in terms of the small parameter $\eps>0$, which measures the criticality of the problem in \eqref{eq:zeta}) is a priori different. Note that the smallness of $\nu, \kappa$ coupled with the presence of a bounda\-ry (and related boundary conditions to be satisfied) is responsible for \emph{the appearance of boundary layers}. It is well-known (see for instance \cite{Metivier, Chemin2006}) that the number and the nature of boundary layers is deeply sensitive to variations of the order of magnitude of viscosity $\nu$ and diffusivity $\kappa$. Our analy\-sis provides a systematic solution to the linear problem \eqref{eq:system} in $\R^2_+$ with boundary conditions \eqref{eq:cond-BL}, under the scaling assumption \eqref{eq:scaling-visc} of \cite{DY1999}  on  the viscosity $\nu>0$ (resp. the diffusivity $\kappa>0$), but without any scaling assumption on the diffusivity coefficient $\kappa>0$ (resp. the viscosity coefficient $\nu>0$). In order to investigate all the possible order of magnitudes (in $\eps$) attained by $\kappa$ (resp. $\nu$), we divide the study in five different cases, which are described later on.
The first four cases that we treat fullfill the scaling assumption \eqref{eq:scaling-visc}. The outcome of our study is that \emph{when the scaling assumption \eqref{eq:scaling-visc} holds, we can always 
 provide a \emph{consistent} and $L^2$ \emph{stable} approximate solution to the linear problem}. Besides, we provide a systematic characterization of boundary layers (boundary layer wave packets) decays and sizes in $L^2(\R^2_+)$ and in $L^\infty (\R^2_+)$.  We further consider the scaling relation where viscosity $\nu$ and diffusivity $\kappa$ have the same size $\nu \sim \kappa \sim \eps^\beta$ with $\beta > 6$, where the case $\beta=6$ was investigated in \cite{BDSR19}. This is outside from the scaling assumption \eqref{eq:scaling-visc}. In this last case, we can provide an approximate solution which is only consistent.
In fact, in this last case we are not able to lift all the three boundary conditions \eqref{eq:cond-BL} because of a \emph{degenerate boundary layer} that we cannot use in our solution. This way, the error that our approximate solution generates on the boundary runs out the possibility of having stability in $L^2$ (because of the presence of the Laplacian, which is incompatibile with any small error on the boundary when one deals with steep function as boundary layers). Note that in this last case $\nu$ and $\kappa$ are allowed to be as small as any (positive integer) power of $\eps$. Therefore this case approaches in some sense the inviscid near-critical reflection problem, which does not have any hope to be stable in $L^2$ (see \cite{BDSR19} for an explanation). \\\\
In the rest of the introduction, we provide a very brief (and very far from being complete) description of the 2d Boussinesq equations, with a non-exhaustive list of mathematical references. The inviscid Boussinesq equations are obtained through a linearization of the density-dependent incompressible Euler equations around special steady states with zero velocity  $(\bar{\rho}_{eq}(x_1), 0, 0, \bar p_{eq}(x_2))$ satisfying the \emph{hydrostatic balance} $\de_{x_2} \bar p_{eq}=-g \bar \rho_{eq}$. In many physical applications, for instance in oceanography, see \cite{Rieutord}, the background density is assumed to be continuous and strictly decreasing $\bar \rho_{eq}'(x_2)<0$: this is called \emph{stable stratification}, as it provides spectral stability of the aforementioned steady states, see \cite{Gallay}. Let us consider then the following perturbation expansions 
\begin{align*}
\rho(t,x_1, x_2)&=\bar{\rho}_{eq}(x_2)+\tilde \rho (t,x_1, x_2); \\
u(t,x_1, x_2)&=\tilde u(t,x_1, x_2);\; w(t,x_1, x_2)=\tilde w (t,x_1,x_2); \\
p(t,x_1,x_2)&=\bar{p}_{eq}(x_2)+\rho_0\tilde p (t,x_1, x_2),
\end{align*} of the hydrostatic equilibria  $(\bar{\rho}_{eq}(x_1), 0, 0, \bar p_{eq}(x_2)), $ where $\bar{\rho}_{eq}(x_2)=\rho_0+r(x_2)$ with $\rho_0>0$ being the characteristic (constant) density and $r(y)$ a function such that $r'(x_2)<0$. Applying the \emph{Boussinesq approximation}, i.e. relying on the assumption that $r(x_2) \ll \rho_0$ and neglecting the density variation $r(x_2)$ everywhere but in gravity terms, \cite{Rieutord}, one obtains system \eqref{eq:Boussinesq-nonrot} with $\nu=\kappa=0$,  after naming $b=g \tilde{\rho}/\rho_0$. In \eqref{eq:Boussinesq-nonrot}, the value $N^2=-g \frac{\bar \rho_{eq}'}{\bar \rho_{eq}}(x_2)$ is called \emph{buoyancy frequency} and it is in general a function of the vertical coordinate $x_2$, depending in fact on the background density profile $\bar \rho_{eq}(x_2)$. When $N$ varies with $x_2$, the linear Boussinesq system is associated with a nonlinear eigenvalue problem that is analyzed in \cite{lannes}. 
However, it is customary especially in oceanography, \cite{DY1999} to restrict ourselves to \emph{affine background stratifications}. In that case, under the Boussinesq approximation $r(x_2) \ll \rho_0$, the value $N$ is constant. We will work in this framework and assume that $N=1$. We finally remark that system \eqref{eq:Boussinesq-nonrot} with constant $N$ has been rigorously derived from the Navier-Stokes Fourier equations under the Oberbeck-Boussinesq approximation (see \cite{long1965}) in several mathematical results, relying on the method of Desjardins et al \cite{desj}: we refer to \cite{danchin} for a recent derivation in critical spaces and for further references therein.

\subsection*{Notation and convention}
We use the following notation and conventions. 
\begin{itemize}
\item We use the notation $a \approx b, a,b \in \R$ if there exist uniform constants $c>0, C>0$ such that $a\ge c b, \, a\le C b$.
\item The notation $a=O(b)$ is used when there exists a uniform constant $C>0$ such that $a \le C b$.
\item For any $f_\eps=f(\eps), g_\eps=g(\eps)=\eps^a\bar g $ with $a \in \R, \bar g \in \R$, we use the notation $f_\eps \sim g_\eps$ if $g_\eps$ is the leading order term of the expansion of $f_\eps$ in terms of $\eps$, i.e. $\lim_{\eps \rightarrow 0} \eps^{-a} f_\eps = \bar g$.
\item Given a function $f$, we denote by $\widehat{f}$ its Fourier transform.
\end{itemize}

%%%%%

\subsection*{Setting of the problem} 
We provide the linear boundary layer analysis of system \eqref{eq:system} with $N=1$, which we rewrite here
\begin{align}
\de_t u - b \sin \gamma + \de_x p &= \nu  \Delta u,\notag\\
\de_t w - b \cos \gamma + \de_y p &=  \nu \Delta w,\notag\\
\de_t b + u\sin \gamma  + w \cos \gamma & = \kappa  \Delta b,\notag\\
\de_x u + \de_y w&=0,
\label{eq:system-sec2}
\end{align}
in $\mathbb{R}^2_+$, with viscosity $\nu>0$ and diffusivity $\kappa>0$, and the following no-slip and no-flux boundary conditions
\begin{align}
u_{|y=0}=w_{|y=0}=\de_y b_{|y=0}=0.\label{eq:cond-BL-sec 2}
\end{align}
We rely on the scaling assumption below.
\begin{assumption}\label{ass}
\begin{itemize}
\item The criticality parameter $\zeta=\omega_{k,m}^2-\sin^2\gamma$ in \eqref{def:critical-relation} is assumed to have the following leading order $\zeta \sim \eps^2$ in terms of an arbitrarily small parameter $\eps>0$.
\item One of the following alternatives holds.
\begin{enumerate}
\item[(A)] The size of viscosity $\nu=\eps^6 \nu_0$ and the size of diffusivity $\kappa=\kappa_0 \eps^\beta$, for some universal constants $\nu_0>0, \kappa_0>0$, and $\beta>0$. 
\item[(B)] The size of diffusivity $\kappa=\eps^6 \kappa_0$ and the size of viscosity $\nu=\nu_0 \eps^\beta$, for some universal constants $\nu_0>0, \kappa_0>0$, and $\beta>0$. 
\item[(C)] The size of viscosity and diffusivity are asymptotically equivalent $\nu \sim \kappa \sim \eps^\beta$, with $\nu= \nu_0 \eps^\beta, \kappa=\kappa_0 \eps^\beta$  for some universal constants $\nu_0>0, \kappa_0>0$, and $\beta>6$. 
\end{enumerate}
\end{itemize}
\end{assumption}
In (C), the sizes of viscosity and diffusivity are not different, in fact they are asymptotically equivalent and strictly smaller than the scaling $\eps^6$ adopted and investigated in \cite{DY1999, BDSR19}. We will show in the following that this last case (C) is \emph{degenerate} and we cannot provide in this context a stable solution to the near-critical reflection problem. 
In order to simplify the analysis, we further distinguish five different (sub-)regimes, which are presented below. Let $\nu_0>0, \kappa_0$ be two universal parameters.
\begin{itemize}
    \item \underline{Case 1}: $\nu=\nu_{0}\eps^{6}$,$\kappa=\kappa_{0}\eps^{\beta}$ with $\beta<6$;
   
    \item \underline{Case 2}: $\nu=\nu_{0}\eps^{6}$,$\kappa=\kappa_{0}\eps^{\beta}$ with $\beta>6$;
    
    \item \underline{Case 3}: $\nu=\nu_{0}\eps^{\beta}$,$\kappa=\kappa_{0}\eps^{6}$ with $\beta<6$;
    
    \item \underline{Case 4}: $\nu=\nu_{0}\eps^{\beta}$,$\kappa=\kappa_{0}\eps^{6}$ with $\beta>6$;
    
    \item \underline{Case 5}: $\nu=\nu_{0}\eps^{\beta}$,$\kappa=\kappa_{0}\eps^{\beta}$ with $\beta>6$. \end{itemize}
\subsection*{Plan of the paper}
The paper is organized as follows. In Section \ref{sec1} we describe and prove the main result, which provides an approximate solution to the considered problem. Next, a detailed linear boundary layer analysis is provided in Section \ref{sec2}. Finally, a more accurate description of the approximate solution is given in the last section.
 
 \section{Main results}\label{sec1}
We state and prove our main result below.
\begin{theorem}\label{thm:main1}[Consistency \& stability]
%Let $\W(t)$ be the unique global-in-time weak solution \bcb{to the system \eqref{eq:system} in $C(\mathbb{R}^+; \mathbb{V}_\sigma')  \cap L^\infty(\mathbb{R}^+; L^2(\mathbb{R}^2_+)) \cap L^2_{\rm{loc}}(\mathbb{R}^+; \mathbb{V}_\sigma)$,  where $\mathbb{V}_\sigma:=\{ (u, w, b) \in H^1(\mathbb{R}^2_+); \; \de_x u+\de_y w=0\}$ and $ \mathbb{V}'_\sigma \rm{\  is\ the\ dual \ of \;} \mathbb{V}_\sigma$,  with initial data
%%{def:inc-beam-norm}
%%\W_{inc}^{0,\sigma,\eps}(t,x,y) 
%$\W_0=(u_0, w_0, b_0) \in \mathbb{V}_\sigma$, $\W_0(x,y)=\W_{inc}^{0,\sigma,\eps}(0,x,y)$ where $\W_{inc}^{0,\sigma,\eps}$ satisfies \eqref{def:inc-beam-norm} (for some $\G$). (\texttt{it must be some conditions of concentration on the initial condition to fit with Definition 1.1, no?})} 
Let $\W_{\rm{inc}}^{0}$ be an incident wave packet satisfying \eqref{def:inc} for some $\G \in C_0^\infty(\R)$.
Then, in all the considered regimes (Case 1,2,3,4,5), there exists an (almost) exact solution $\W^{\rm{app}}$ to system \eqref{eq:system}, of the form
$$\W^{\rm{app}}:=\W^0_{\rm{inc}}+\W^0_{\rm{BL}},$$
where  $\W_{\rm{BL}}^0$ is a boundary layer wave packet as in \eqref{def:BL-general}. 
More precisely, the following results hold true.
\begin{itemize}
\item[(i)] In all the considered regimes (Case 1,2,3,4,5),  $\W^{\rm{app}}$ is a \emph{consistent} approximation to system \eqref{eq:system}, in the sense that it satisfies that system 
\begin{align*}
\de_t u^{\rm{app}} - b^{\rm{app}} \sin \gamma + \de_x p^{\rm{app}} &=  \nu  \Delta u^{\rm{app}}  + R_u^{\rm{app}},\notag\\
\de_t w^{\rm{app}}  - b^{\rm{app}}  \cos \gamma + \de_y p^{\rm{app}}  &= \nu \Delta w^{\rm{app}}  + R_w^{\rm{app}} ,\notag\\
\de_t b^{\rm{app}}  + u^{\rm{app}} \sin \gamma  + w^{\rm{app}}  \cos \gamma & =  \kappa  \Delta b^{\rm{app}} + R_b^{\rm{app}} ,\notag\\
\de_x u^{\rm{app}}  + \de_y w^{\rm{app}} &=0,
\end{align*}
and the boundary conditions \eqref{eq:cond-BL}, with a remainder $$R_{\rm{app}}=(R_u^{\rm{app}} , R_w^{\rm{app}} , R_b^{\rm{app}})^T \quad \text{such that} \quad \|R_{\rm{app}}\|_{L^2(\R^2_+)}=O(\nu+\kappa).$$
\item[(ii)] In Case 1,2,3,4, $\W^{\rm{app}}$ is a \emph{stable} approximation to system \eqref{eq:system} in the following sense. Consider the unique global-in-time weak solution $\W_{\rm{weak}} (t)$ to system \eqref{eq:system} with boundary conditions \eqref{eq:cond-BL} in $C(\mathbb{R}^+; \mathbb{V}_\sigma')  \cap L^\infty(\mathbb{R}^+; L^2(\mathbb{R}^2_+)) \cap L^2_{\rm{loc}}(\mathbb{R}^+; \mathbb{V}_\sigma)$,  where $\mathbb{V}_\sigma:=\{ (u, w, b) \in H^1(\mathbb{R}^2_+); \; \de_x u+\de_y w=0\}$ and $ \mathbb{V}'_\sigma \rm{\  is\ the\ dual \ of \;} \mathbb{V}_\sigma$,  with initial data $\W_{\rm{weak}}|_{t=0}=\W^{\rm{app}}|_{t=0}$.
Then, the following estimate holds:
\begin{align}\label{eq:stabestimate}
\|(\W^{\rm{app}}-\W_{\rm{weak}})(t\bcb{,\cdot,\cdot})\|_{L^2(\mathbb{R}^2)} = O(\sqrt{(\nu+\kappa)t})e^{\max\{\nu, \kappa\} t}.
\end{align}
\end{itemize}
\end{theorem}

We will rely on the following (classical) result to prove \emph{stability} of our approximate solution $\W^{\rm{app}}$.
\begin{theorem}[On the existence of weak solutions in the half plane]
Let $\W_{0}=(u_{0},v_{0},b_{0}) \in (L^2(\R^2_+))^2$ be a divergence-free initial datum. Then there exists a unique global-in-time weak solution 
\[\W_{\rm{weak}} \in C(\mathbb{R}^+; \mathbb{V}_\sigma')  \cap L^\infty(\mathbb{R}^+; L^2(\mathbb{R}^2_+)) \cap L^2_{\rm{loc}}(\mathbb{R}^+; \mathbb{V}_\sigma)\] 
to system \eqref{eq:system} with boundary condition \eqref{eq:cond-BL}, which satisfies the following energy identity for all $t\ge0$:
\begin{equation}
\frac{1}{2}\|\W_{\rm{weak}}(t)\|^{2}_{L^{2}}+ \int^{T}_{0} \nu (\|\nabla u(t)\|^{2}_{L^{2}}+\|\nabla w(t)\|^{2}_{L^{2}})+\kappa\|\nabla b(t)\|^{2}_{L^{2}}\, dt= \frac{1}{2}\|\mathcal{W}_{0}(0)\|^{2}_{L^{2}}.
\label{eq:energy}
\end{equation}
\end{theorem}
The existence of a unique global-in-time weak solution $\W_{\rm{weak}}$ is classical and it follows by applying the Spectral Theorem to the inverse of the Stokes operator acting on the half plane. It relies on the construction of a sequence of regularized solutions $\W_k(t)$, for which we can derive the energy identity. The proof is detailed in \cite{Chemin2006} and in the Appendix of \cite{BDSR19}. 

\begin{proof}[Proof of Theorem \ref{thm:main1}]
We prove (i) and (ii) respectively.
\begin{itemize}
\item[(i)] The solution $\W^{\rm{app}}$ is provided by Proposition \ref{prop:BL}, together with the estimate of $\|R_{\rm{app}}\|_{L^2}$. The \emph{consistency} part is therefore proved.\\
\item[(ii)] We point out that the strong formulation of the system, which is required in order to get the energy estimate leading to the stability inequality, is actually satisfied by a sequence approximate solutions, which are smooth by Friedrichs approximation. Then the energy inequality for the weak solution $\W_{\rm{weak}}$ is obtained by passing to the limit. With a slight abuse of notation, here we omit this step. We write the equation satisfied by $\W^{\rm{app}}-\W_{\rm{weak}}$ below:
\begin{align}\label{eq:BQ-remainder}
\partial_{t}(\W^{\rm{app}}-\W_{\rm{weak}})=-\mathbb{P}\mathcal{L}(\W^{\rm{app}}-\W_{\rm{weak}})+\mathbb{P}
\begin{pmatrix}
   \nu\Delta (u^{\rm{app}}-u_{\rm{weak}})\\ \nu\Delta (w^{\rm{app}}-w_{\rm{weak}}) \\ \kappa\Delta (b^{\rm{app}}-b_{\rm{weak}}) 
\end{pmatrix} + R^{\rm{app}},
\end{align}
where $\mathbb{P}$ is the Leray projector such that $\mathbb{P} \W_{\rm{weak}}= \W_{\rm{weak}}$ and  $\mathbb{P} \W^{\rm{app}}= \W^{\rm{app}}$, and  $\mathcal{L}$ is the following skew-symmetric operator:
\[
\mathcal{L}=
\begin{pmatrix}
   0 & 0 & -\sin\gamma \\
    0 & 0 & -\cos\gamma \\
    \sin\gamma & \cos\gamma & 0
\end{pmatrix}.
\]
After taking the scalar product against $\W^{\rm{app}}-\W_{\rm{weak}}$ in \eqref{eq:BQ-remainder}, we obtain 
\begin{align} \label{eq:scalarboussinesq}
&\frac{1}{2}\frac{d}{dt}\|\W^{\rm{app}}-\W_{\rm{weak}}\|_{L^{2}}^2+\langle\mathbb{P}\mathcal{L}(\W^{\rm{app}}-\W_{\rm{weak}}),\W^{\rm{app}}-\W_{\rm{weak}}\rangle\\ 
& \quad =\left\langle\begin{pmatrix}
   \nu\Delta (u^{app}-u_{\rm{weak}})\\ \nu\Delta (w^{app}-w_{\rm{weak}})\notag \\ \kappa\Delta (b^{app}-b_{\rm{weak}}) 
\end{pmatrix},\W^{\rm{app}}-\W_{\rm{weak}}\right\rangle+\langle R^{\rm{app}},\W^{\rm{app}}-\W_{\rm{weak}}\rangle.
\end{align}
Since $\mathbb{P}$ is symmetric while $\mathcal{L}$ is skew-symmetric, we have
\begin{align*}\langle\mathbb{P}\mathcal{L}(\W^{\rm{app}}-\W_{\rm{weak}}),\W^{\rm{app}}-\W_{\rm{weak}}\rangle&=\langle\mathcal{L}(\W^{\rm{app}}-\W_{\rm{weak}}),\mathbb{P}\W^{\rm{app}}-\W_{\rm{weak}}\rangle\\
&=\langle\mathcal{L}(\W^{\rm{app}}-\W_{\rm{weak}}),\W^{\rm{app}}-\W_{\rm{weak}}\rangle\\
&=0.
\end{align*}
Integrating by parts the terms involving the Laplace operator, and using the fact that the approximate solution exactly satisfies the boundary conditions, i.e. $\W^{\rm{app}}|_{y=0}=0$, we have
\begin{align*}
&-\nu\int_\RR \partial_y (u^{\rm{app}}-u_{\rm{weak}})(u^{\rm{app}}-u_{\rm{weak}})|_{y=0}+
\partial_y (w^{\rm{app}}-w_{\rm{weak}})(w^{\rm{app}}-w_{\rm{weak}})|_{y=0}\, dx\\
&\qquad -\kappa\int_\RR \partial_y (b^{\rm{app}}-b_{\rm{weak}})(b^{\rm{app}}-b_{\rm{weak}})|_{y=0} \, dx\\
&\qquad -\nu(\|\nabla(u^{\rm{app}}-u_{\rm{weak}})\|_{L^2}+\|\nabla(w^{\rm{app}}-w_{\rm{weak}})\|_{L^2})-\kappa\|\nabla(b^{\rm{app}}-b_{\rm{weak}}\|_{L^2}\\
&\quad =-\nu(\|\nabla(u^{\rm{app}}-u_{\rm{weak}})\|_{L^2}+\|\nabla(w^{\rm{app}}-w_{\rm{weak}})\|_{L^2})-\kappa\|\nabla(b^{\rm{app}}-b_{\rm{weak}}\|_{L^2}\le0.
\end{align*}

In the last term of \eqref{eq:scalarboussinesq}, we decompose $R_{\rm{app}} = R^{{\rm{app}}}_\nu + R^{{\rm{app}}}_\kappa$, where we know from Proposition \ref{prop:BL} that $\|R^{{\rm{app}}}_\nu\|_{L^2} = O (\nu)$ and $\|R^{{\rm{app}}}_\kappa\|_{L^2}=O(\kappa)$, and we obtain
\begin{align*}
|\langle R_{{\rm{app}}},\W^{\rm{app}}-\W_{\rm{weak}}(t)\rangle| \leq &\frac{1}{2\nu} \| R^{\rm{app}}_\nu \|_{L^{2}}^2 + \frac{1}{2\kappa} \| R^{{\rm{app}}}_\kappa \|_{L^{2}}^2+ \nu\| (u^{{\rm{app}}}-u)(t)\|^{2}_{L^{2}}\\
&\quad +\nu\| (w^{{\rm{app}}}-w)(t)\|^{2}_{L^{2}}+ \kappa  \|(b^{{\rm{app}}}-b)(t)\|^{2}_{L^{2}}.
\end{align*}
Finally, putting altogether inside \eqref{eq:scalarboussinesq}, we have
\begin{align}\notag
\frac{1}{2}\frac{d}{dt}\|\W^{\rm{app}}-\W_{\rm{weak}}\|_{L^{2}}^2 \le \frac{1}{2\nu} \| R^{{\rm{app}}}_\nu \|_{L^{2}}^2 + \frac{1}{2\kappa} \| R^{{\rm{app}}}_\kappa \|_{L^{2}}^2 + \max\{\nu, \kappa\} \|\W^{\rm{app}}-\W_{\rm{weak}}\|_{L^{2}}^2.
\end{align}

Since $(\W^{\rm{app}}-\W_{\rm{weak}})|_{t=0}=0,$ and, as remarked before, we have
$$\nu^{-1} \|R^{{\rm{app}}}_\nu\|^2_{L^2(\mathbb{R}_+^2)}+ \kappa^{-1} \|R^{{\rm{app}}}_\kappa\|^2_{L^2(\mathbb{R}_+^2)}  = O(\nu+\kappa),$$
integrating in time and using the Gr\"onwall estimate, we obtain 
\begin{align}\notag
\|(\W^{\rm{app}}-\W_{\rm{weak}})(t\bcb{,\cdot,\cdot})\|_{L^{2}} \le O(\sqrt{(\nu+\kappa) t}) e^{\max\{\nu, \kappa\} t}.
\end{align}  
The proof is concluded.
\end{itemize}
\end{proof}

\section{Linear boundary layer analysis}\label{sec2}
In this section, we construct the approximate solution $\W^{\rm{app}}$ to the near-critical reflection problem for \eqref{eq:system-sec2}, as announced in the statement of Theorem \ref{thm:main1}. The approximate solution $\W^{\rm{app}}$ is provided by the following result.

\begin{proposition}\label{prop:BL}[Critical reflection for an incident wave packet]
Let $\nu>0, \kappa>0$. 
There exists an approximate solution $(u, w, b)^T$ to system \eqref{eq:system} in the half space $\mathbb{R}^2_+$, which \textbf{exactly satisfies the boundary conditions} \eqref{eq:cond-BL}, and it is given by
\begin{align*}
\W^{\rm{app}}:=\W_{\rm{inc}}^0 + \Wbl^0 ,  
\end{align*}
where $\W_{\rm{inc}}^0=(u_{\rm{inc}}, w_{\rm{inc}}, b_{\rm{inc}})^T$ is the incident wave packet \eqref{def:inc}, while $\W_{\rm{BL}}^0$ is a boundary layer wave packet. 
In particular, the boundary layer wave packet $\Wbl^0$ has the following form in the different regimes.
\begin{itemize}
\item  In \underline{Case 1} ($\nu\sim\eps^{6},\kappa\sim\eps^{\beta}, \beta<6$):
$$\W^0_{\rm{BL}}:=\W^0_{\rm{BL}, \eps^{3}}+\W^0_{\rm{BL}, \eps^{\beta/3}},$$
where $\W^0_{\rm{BL}, \eps^{\beta/3}}$ is a boundary layer wave packet of decay $\eps^{-\beta/3}$ and amplitude $\eps^{-\beta/3}$, $\W^0_{\rm{BL}, \eps^{3}}$ is a boundary layer wave packet of decay $\eps^{-3}$ and amplitude $\eps^{\beta/3-3}$ for $\beta>9/2$ and $\eps^{-\beta/3}$ for $\beta<9/2$.
\item In \underline{Case 2} ($\nu\sim\eps^{6},\kappa\sim\eps^{\beta}, \beta>6$): 
$$\W^0_{\rm{BL}}:=\W^0_{\rm{BL}, \eps^{2}}+\W^0_{\rm{BL}, \eps^{\beta/2}},$$
where $\W^0_{\rm{BL}, \eps^{2}}$ is a boundary layer wave packet of decay $\eps^{-2}$ and amplitude $\eps^{-2}$, $\W^0_{\rm{BL}, \eps^{\beta/2}}$ is a boundary layer wave packet of decay $\eps^{-\beta/2}$ and amplitude $\eps^{3\beta/2-10}$.
\item In \underline{Case 3} ($\nu\sim\eps^{\beta},\kappa\sim\eps^{6}, \beta<6$): 
$$\W^0_{\rm{BL}}:=\W^0_{\rm{BL}, \eps^{\beta/3}}+\W^0_{\rm{BL}, \eps^{3}},$$
where $\W^0_{\rm{BL}, \eps^{\beta/3}}$ is a boundary layer wave packet of decay $\eps^{-\beta/3}$ and amplitude $\eps^{-\beta/3}$, $\W^0_{\rm{BL}, \eps^{3}}$ is a boundary layer wave packet of decay $\eps^{-3}$ and amplitude $\eps^{9-5\beta/3}$.
\item In \underline{Case 4} ($\nu\sim\eps^{\beta},\kappa\sim\eps^{6}, \beta>6$):
$$\W^0_{\rm{BL}}:=\W^0_{\rm{BL}, \eps^{2}}+\W^0_{\rm{BL}, \eps^{\beta/2}},$$
where $\W^0_{\rm{BL}, \eps^{2}}$ is a boundary layer wave packet of decay $\eps^{-2}$ and amplitude $\eps^{-2}$, $\W^0_{\rm{BL}, \eps^{\beta/2}}$ is a boundary layer wave packet of decay $\eps^{-\beta/2}$ and amplitude $\eps^{2-\beta/2}$ for $\beta<8$ and $\eps^{-2}$ for $\beta>8$.
\item In \underline{Case 5} ($\nu\sim\eps^{\beta},\kappa\sim\eps^{\beta}, \beta>6$):
$$\W^0_{\rm{BL}}:=\W^0_{\rm{BL}, \eps^{2}}+\W^0_{\rm{BL}, \eps^{\beta/2-1}}+\W^0_{\rm{BL}, \eps^{\beta/2}},$$
where $\W^0_{\rm{BL}, \eps^{-\beta+8}}$ is a boundary layer wave packet of decay $\eps^{\beta-8}$ and amplitude $\eps^{-2}$, $\W^0_{\rm{BL}, \eps^{\beta/2-1}}$ is a boundary layer wave packet of decay $\eps^{1-\beta/2}$ and amplitude $\eps^{-2}$ and $\W^0_{\rm{BL}, \eps^{\beta/2}}$ is a boundary layer wave packet of decay $\eps^{-\beta/2}$ and amplitude $\eps^{-1}$.
\end{itemize}

Moreover, the following results hold true.
\begin{itemize}
\item[(a)] The boundary layer wave packet $\Wbl^0=(u_{\rm{BL}}, w_{\rm{BL}}, b_{\rm{BL}})^T$ is an \textbf{exact solution} to system \eqref{eq:system-sec2} with boundary conditions 
\begin{align}\label{eq:BL-prop}
u_{\rm{BL}}|_{y=0}= - u_{\rm{inc}}|_{y=0}, \quad w_{\rm{BL}}|_{y=0}= - w_{\rm{inc}}|_{y=0}, \quad \de_y b_{\rm{BL}}|_{y=0}= - \de_y b_{\rm{inc}}|_{y=0}.
\end{align}
\item[(b)] The function $\W^{\rm{app}}$ is a \textbf{consistent} approximate solution in the sense of Theorem \ref{thm:main1}, with a remainder $R_{\rm{app}}=(R_u^{\rm{app}} , R_w^{\rm{app}} , R_b^{\rm{app}})^T$ such that $$\|R_{\rm{app}}\|_{L^2(\R^2_+)}\le \|R_u^{\rm{app}}\|_{L^2(\R^2_+)} + \|R_w^{\rm{app}}\|_{L^2(\R^2_+)} + \|R_b^{\rm{app}}\|_{L^2(\R^2_+)} = O(\nu+\kappa),$$ where $\|R_u^{\rm{app}}\|_{L^2(\R^2_+)}=O(\nu), \|R_w^{\rm{app}}\|_{L^2(\R^2_+)}=O(\nu), \|R_b^{\rm{app}}\|_{L^2(\R^2_+)}=O(\kappa)$. More precisely, in the different cases, we have the following:
\begin{table}[ H]
\begin{tabular}{|p{3.3cm} |p{1.4cm} |p{1.4cm} |p{1.4cm} |p{1.4cm} |p{1.4cm} |}
 \hline
 & Case 1 & Case 2 & Case 3 & Case 4 & Case 5  \\ [1.5ex] \hline
$\|R^{\rm{app}}\|_{L^{2}}$ &   $O(\eps^{\beta})$     &    $O(\eps^{6})$    &    $O(\eps^{\beta})$    &    $ O(\eps^{6})$ & $O(\eps^\beta)$  \\ [3ex] \hline
\end{tabular}
\end{table}
\end{itemize}
\end{proposition}

The first step to construct $\W^{\rm{app}}$ in the above statement is to determine the precise order of the boundary layers solving the linear system \eqref{eq:system}. We will rely on the following definition.
%%%%%
\begin{definition}\label{def:BL}
For any $\beta_j >0$, $\alpha_j \in \RR$, let us define a \emph{boundary layer wave packet} as follows:
\begin{align}\label{def:BL-general}
\W_{\rm{BL}, \eps^{\beta_j}}^{\alpha_j}:= \int_{\R^2} \widehat{A}(k,m) a_j X_{k, \L} e^{-i\omega t+ikx - \L y} \, dk \, dm, \quad \rm{Re}(\L)>0,
\end{align}
where
\begin{enumerate}
\item $\lambda = O(\eps^{-\beta_j})$; 
\item $a_j =O(\eps^{\alpha_j})$;
\item $\widehat{A}(k,m)$ is given in \eqref{def:psi}; 
\item the eigenvector \begin{align}
\R^3 \ni X_{k, \lambda}&=\begin{pmatrix}
U_\L\\
W_\L\\
B_\L\\
\end{pmatrix}, \quad \text{with \;} U_\L=U_\L(k), \, W_\L=W_\L(k), \,B_\L=B_\L(k);
\end{align}
\item the time frequency $\omega_{k,m}$ is given by \eqref{eq:disprel}.
\end{enumerate}
\end{definition}

Plugging the ansatz $\eqref{def:BL-general}$ inside the equations \eqref{eq:system}, one obtains the corresponding (algebraic) linear system

\begin{align}\label{eq:algebra}
A_\eps(\omega, \kappa, \nu, k, \L) \begin{pmatrix}
U_\L\\
W_\L\\
B_\L\\
P_\L
\end{pmatrix}=0, \quad \text{where \;} A_\eps(\omega, \kappa, \nu, k, \L) \in \, \text{space of matrices \;} \mathcal{M}^{4\times 4},
\end{align}
$a_j=a_j(k, \eps)$ is the amplitude of the boundary layer.
We look therefore for vectors $X_{k, \lambda} \in \ker (A_\eps(\omega, \kappa_0, \nu_0, k, \L))$, with the restriction that $\lambda=\lambda(k)$ is such that $\rm{Re}(\lambda)>0$, as \eqref{def:BL-general}. Such vectors $X_{k, \L}$ will be called hereafter \emph{eigenvectors associated to boundary layers}. In order that $\ker A_\eps(\omega, \kappa_0, \nu_0, k, \L) \neq \{0\}$, one asks $\det A_\eps(\omega, \kappa_0, \nu_0, k, \L)=0$. This amounts at finding the roots in $\L$ of the following 
characteristic polynomial associated with $A_\eps(\omega, \kappa_0, \nu_0, k, \L)$:
\begin{align}\label{eq:pol}
\mathcal{P}(\L)&:=-\kappa \nu \lambda^6 - (i\omega (\kappa+\nu)+3\nu \kappa k^2)\L^4  + (\zeta  +2i\omega (\kappa+\nu) k^2-3\nu\kappa k^4) \L^2\notag\\
&\quad  - 2i \L k \sin\gamma \cos \gamma + k^2 (\cos^2\gamma-\omega^2-i\omega (\kappa+\nu) k^2+\nu\kappa k^4). 
\end{align}
Furthermore, the eigenvector $X_{k, \L_j}$ related to the eigenvalue $\L_j$ (such that $\mathcal{P}(\L_j)=0$) has the following general form:
\begin{align}\label{eq:eigen-BLs}
X_{k, \L_j}&=\begin{pmatrix}
U_{\L_j}\\
W_{\L_j}\\
B_{\L_j}
\end{pmatrix}= \begin{pmatrix}
1 \\
 \frac{ik}{\L_j} \\
\frac{\sin \gamma + i k \L_{j}^{-1} \cos \gamma}{i\omega - \kappa (k^2-\L_j^2)} \\
\end{pmatrix},
\end{align}
while the pressure $P_{\L_j}=\frac{1}{ik} [i\omega + \nu(\L^2-k^2) + \sin\gamma \frac{\sin \gamma+ik\L_{j}^-1 \cos \gamma}{i\omega+\kappa (\L^2-k^2)}]$.

The rest of this section is devoted to the proof of Proposition \ref{prop:BL}.
\begin{proof}[Proof of Proposition \ref{prop:BL}]
Notice that since wave packets are simple superposition of plane waves, the following linear analysis can be performed indeed in terms of plane waves. First, we provide the asymptotics of the roots $\L$ of $\mathcal{P}(\L)$ in the different cases. Note that we need at least three roots $\L$ with $\text{Re}(\L)>0$ (three boundary layers) in order to lift the boundary conditions \eqref{eq:cond-BL}. Next, we determine the amplitudes $a_j$ by requiring that the sum of the boundary layers evaluated at $y=0$ balances the boundary contribution of the incident wave (crf. (a) in the statement of Proposition \ref{prop:BL}.)\\
In order to determine the leading order of the roots  $\L$ of $\mathcal{P}(\L)$, we apply the same method that is widely detailed in \cite{BDSR19} and we look for asymptotics of the type $\L\sim\bar{\L}(k)\eps^{q}$ where $\bar{\L}(k) \in \CC$ and $q \in \QQ$. To determine their asymptotic behavior, we need to find $q$ such that different monomials in the expression of $\mathcal{P}(\L)$ have the same order in terms of (powers of) $\eps$. This way, we find an equation in $\bar{\L}$ that we can solve.
First, it is easy to see that if there exists a constant $c>0$ such that $1/c\le|k|\le c, \;1/c\le|\cos^2-\omega^2|\le c $ in all the cases, then there is no any root with $q>0$ (i.e an infinitesimal root in $\eps$), while there is always a root of size 1. 
In particular, the root of size 1 is 
$$ \lambda_{1} \sim -i\frac{k(\cos^{2}\gamma-\omega^{2})}{2\sin\gamma\cos\gamma},$$
and this root corresponds to the incident wave packet.
Now we investigate the different cases in detail. Hereafter we drop the \emph{overline} in $\bar \L$ for lightening the notation.

\noindent \underline{Case 1} ($\nu\sim\eps^{6},\kappa\sim\eps^{\beta}, \beta<6$):
\begin{itemize}
    \item for $q=-\frac{\beta}{3}$ we have
    \[-i\omega\kappa\lambda^{3}-2ik\sin\gamma\cos\gamma=0.\]
Then there are two roots, denoted by $\lambda_{2},\lambda_{3}$, with positive real part, and another one, $\lambda_{4}$, with negative real part.
\item For $q=-3$, we have
\[-\nu\lambda^{2}-i\omega=0,
\]
then there is a root, $\lambda_{5}$, with positive real part and another one, $\lambda_{6}$, with negative real part.
\end{itemize}

\noindent \underline{Case 2} ($\nu\sim\eps^{6},\kappa\sim\eps^{\beta}, \beta>6$):
\begin{itemize}
\item For $q=-2$, we have
\[-i\omega\nu\lambda^{3}+\zeta\lambda-2ik\sin\gamma\cos\gamma=0.
\]
There are two roots with positive real part, $\lambda_{2},\lambda_{3}$, and one root, $\lambda_{4}$, with negative real part.
\item For $q=-\frac{\beta}{2}$ we have
\[-\kappa\lambda^{2}-i\omega=0,\]
there are two root $\lambda_{5},\lambda_{6}$, the first one with positive real part and the second one with negative real part.
\end{itemize}
\begin{remark}
The next two cases are very similar to the previous ones since  \eqref{eq:pol} is symmetric in $(\nu, \kappa)$.
\end{remark}

\noindent \underline{Case 3} ($\nu\sim\eps^{\beta},\kappa\sim\eps^{6}, \beta<6$):
\begin{itemize}
    \item for $q=-\frac{\beta}{3}$ we have
    \[-i\omega\nu\lambda^{3}-2ik\sin\gamma\cos\gamma=0.\]
There are two $\lambda$, denoted with $\lambda_{2},\lambda_{3}$, with positive real part, and another one, $\lambda_{4}$, with negative real part.
\item For $q=-3$ we have
\[-\kappa\lambda^{2}-i\omega=0,
\]
there is a root, $\lambda_{5}$, with positive real part and another one, $\lambda_{6}$, with negative real part.
\end{itemize}

\noindent \underline{Case 4} ($\nu\sim\eps^{\beta},\kappa\sim\eps^{6}, \beta>6$):
\begin{itemize}
\item For $q=-2$ we have
\[-i\omega\kappa\lambda^{3}+\zeta\lambda-2ik\sin\gamma\cos\gamma=0,
\]
there are two roots with positive real part, $\lambda_{2},\lambda_{3}$, and one, $\lambda_{4}$, with negative real part.
\item For $q=-\frac{\beta}{2}$ we have
\[-\nu\lambda^{2}-i\omega=0,\]
there are two roots $\lambda_{5},\lambda_{6}$, the first one with positive real part and the second one with negative real part.
\end{itemize}
\noindent \underline{Case 5} ($\nu\sim\kappa\sim\eps^{\beta}$, $\beta>6$):
\begin{itemize}
\item For $q=-2$ we have
\[\zeta\lambda-2ik\sin\gamma\cos\gamma=0.\]
There is one root $\L_{2}$ with positive real part. In fact, the first approximation of this root is purely imaginary, but a further expansion yields
$$ \L_{2}=\bar{\L}\eps^{-2}+\frac{i\omega(\kappa+\nu)\bar{\L}^{4} \eps^{-8}}{2ik\sin\gamma\cos\gamma+O(\eps^{\delta})}-\frac{k^{2}(\cos^{2}\gamma-\omega^{2})}{2ik\sin\gamma\cos\gamma+O(\eps^{\delta})}+O(\eps^{s}), \quad s,\delta>0.$$
Therefore for $\beta<8$ this root has positive real part of order $\eps^{\beta-8}$. Since for $\beta>8$ this root is infinitesimal with respect to $\eps$, the associated part of solution does not represent a boundary layer in the sense of Definition \ref{def:BL}: it will be called \emph{degenerate boundary layer} (in fact this root has strictly positive real part for $\beta \ge 8$, but its decay in $y$ is very slow, i.e. it is a positive - rather than a negative - power of $\eps$) .
%P'(\lambda)=-6\nu\kappa\lambda^{5}-4i\omega(\kappa+\nu)\lambda^{3}+2ik\sin\gamma\cos\gamma+2ik(\kappa+\nu)k^{2}\lambda
\item For $q=-\frac{\beta-2}{2}$ we have 
\[-i\omega(\nu+\kappa)\lambda^{2}+\zeta=0,
\]
where regardless of $\zeta$ there are always two roots $\lambda_{3},\lambda_{4}$, one with positive real part, $\lambda_{3}$, and another one with negative real part $\lambda_{4}$.
\item For $q=-\frac{\beta}{2}$ we have
\[-\nu\kappa\lambda^{2}-i\omega(\nu+\kappa)=0.\]
where $\lambda_{5},\lambda_{6}$ are similar to the previous cases.
\end{itemize}

\begin{remark}
As remarked in \cite{BDSR19}, the number of the roots with positive real part is independent of $\zeta$. Therefore, in all the regimes we always have three roots with positive real part.
\end{remark}

\noindent We summarize in the next table what we have found so far:
\begin{table}[H]
\begin{tabular}{ |p{2.4cm} |p{2.4cm} |p{2.4cm} |p{2.4cm} |p{2.8cm}| }
\hline
Case 1 & Case 2 & Case 3 & Case 4 & Case 5 \\ [1.5ex] \hline
One BL $\nu^{\frac{1}{2}}$  & One BL $\kappa^{\frac{1}{2}}$ & One BL $\kappa^{\frac{1}{2}}$   & One BL $\nu^{\frac{1}{2}}$ & One BL $\zeta^{4}/\nu$   \\[3ex]
 One BL $\kappa^{\frac{1}{3}}$  & One BL $\nu^{\frac{1}{3}}$   & One BL $\nu^{\frac{1}{3}}$ & One BL $\kappa^{\frac{1}{3}}$   &  One  BL $\nu^{\frac{1}{2}}$\\ [3ex]
 &  &  &  & One BL $(\nu/\zeta)^{\frac{1}{2}}$  \\[3ex] \hline
\end{tabular}
\end{table}

\underline{We are ready to prove point (a) of Proposition \ref{prop:BL}.}
To this end, we have to ensure that all the boundary conditions can be lifted (i.e., to guarantee that there exist always three roots $\L_j$ of $\mathcal{P}(\L)$ with $\Re(\L_j)>0$) and to determine the amplitudes of the related boundary layer wave packets. To this end, we have to solve the following linear algebraic system for $(a_{2},a_3,a_5)$ (see the notation for $X_{k,m}$ in \eqref{eq:disprel})
\begin{equation}\label{eq:amplitudeSystem}
M_{\lambda}(k,\eps,\omega)\begin{pmatrix}
a_{2} \\
a_{3} \\
a_{5}
\end{pmatrix}=
\begin{pmatrix}
\mathfrak{u}\\
\mathfrak{w}\\
\mathfrak{b}
\end{pmatrix}, \quad \text{with} \quad 
\begin{pmatrix}
\mathfrak{u}\\
\mathfrak{w}\\
\mathfrak{b}
\end{pmatrix}=
\begin{pmatrix}
- \widehat{u}_{\rm{inc}}\\
- \widehat{w}_{\rm{inc}}\\
- im (\widehat{b}_{\rm{inc}})
\end{pmatrix},
\end{equation} 
where $(a_{2},a_3,a_5) \in \CC^3$ are the unknown amplitudes and
\begin{equation*}
    M_{\lambda}(k,\eps,\omega)=\begin{pmatrix}U_{\lambda_{2}} & U_{\lambda_{3}} & U_{\lambda_{5}} \\
  W_{\lambda_{2}} & W_{\lambda_{3}} & W_{\lambda_{5}}\\
 -\lambda_{2} B_{\lambda_{2}} & -\lambda_{3} B_{\lambda_{3}} & -\lambda_{5} B_{\lambda_{5}} \end{pmatrix}=\begin{pmatrix}1 & 1 & 1 \\
  \frac{ik}{\lambda_{2}} & \frac{ik}{\lambda_{3}} & \frac{ik}{\lambda_{5}}\\
 -\lambda_{2} B_{\lambda_{2}} & -\lambda_{3} B_{\lambda_{3}} & -\lambda_{5} B_{\lambda_{5}} \end{pmatrix},
\end{equation*}
with $(U_{\L_j}, W_{\L_j}, B_{\L_j})$ given by \eqref{eq:eigen-BLs}.
It turns out that the number of boundary conditions which can be lifted is exactly the dimension of the following vector space
\begin{align*}
\text{Vect}\{(U,W,-\lambda B) \in \mathbb{C}^{3} \; | \; \exists (P,\lambda) \in \mathbb{C}^{2} , \Re(\lambda)>0 \,\text{e} \, &P_{\nu,\kappa,\omega,k}(\lambda)=0 \\
&\text{s. t.} \;  (U,W,B,P) \in \text{ker}A_\eps(\omega, \kappa, \nu, k, \L)\}.
\end{align*}
We can check that this system always has non-trivial solutions (i.e. det$M_{\lambda}(k,\eps,\omega)\neq0$), so that we have
\begin{equation} \label{eq:detMatrix}
    \text{det}M_\L(k,\eps,\omega)=ik\left( B_{\lambda_{2}}\frac{\lambda_{2}}{\lambda_{3}}-B_{\lambda_{3}}\frac{\lambda_{3}}{\lambda_{2}}+B_{\lambda_{3}}\frac{\lambda_{3}}{\lambda_{5}}-B_{\lambda_{2}}\frac{\lambda_{2}}{\lambda_{5}}+B_{\lambda_{5}}\frac{\lambda_{5}}{\lambda_{2}}-B_{\lambda_{5}}\frac{\lambda_{5}}{\lambda_{3}}\right).
\end{equation}
Inverting the matrix $M_{\lambda}(k,\eps,\omega)$ and observing that $\mathfrak{u},\mathfrak{v},\mathfrak{w}$ are all of order $1$, the order of $a_j$ is given by the leading order of the corresponding row of $M_{\lambda}(k,\eps,\omega)^{-1}$. We provide more details below.
First, notice that from \eqref{eq:eigen-BLs} that for $j=2,3$ we have $B_{\L_j}=O(1)$, and (if det$M_{\lambda}(k,\eps,\omega)\neq0$)
\begin{equation*}
   M_{\lambda}(k,\eps,\omega)^{-1}=\frac{1}{\text{det}M_{\lambda}(k,\eps,\omega)}
    \begin{pmatrix}
       -ikB_{\lambda_{5}}\frac{\lambda_{5}}{\lambda_{3}}+ikB_{\lambda_{3}}\frac{\lambda_{3}}{\lambda_{5}} & B_{\lambda_{5}}\lambda_{5}-B_{\lambda_{3}}\lambda_{3} & ik(\frac{1}{\lambda_{5}}-\frac{1}{\lambda_{3}}) \\
ikB_{\lambda_{5}}\frac{\lambda_{5}}{\lambda_{2}}-ikB_{\lambda_{2}}\frac{\lambda_{2}}{\lambda_{5}} & -B_{\lambda_{5}}\lambda_{5}+B_{\lambda_{2}}\lambda_{2} & ik(\frac{1}{\lambda_{2}}-\frac{1}{\lambda_{5}})\\

-ikB_{\lambda_{2}}\frac{\lambda_{2}}{\lambda_{3}}+ikB_{\lambda_{3}}\frac{\lambda_{3}}{\lambda_{2}} & B_{\lambda_{3}}\lambda_{3}-B_{\lambda_{2}}\lambda_{2} & ik(\frac{1}{\lambda_{3}}-\frac{1}{\lambda_{2}}) \\
    \end{pmatrix}.
\end{equation*}

Now we need to distinguish the five cases. 

$\bullet$ \noindent\underline{Case 1} $(\nu\sim\eps^6,\kappa\sim\eps^\beta, \beta<6)$:
 recalling that $\L_5=O(\eps^{-3})$ we have
\[ B_{\lambda_{5}}=\frac{\sin \gamma + i k \L_{5}^{-1} \cos \gamma}{i\omega - \kappa (k^2-\L_5^2)}=\frac{O(1)+O(\eps^{3})}{O(1)+O(\eps^{\beta})+O(\eps^{\beta-6})}=O(\eps^{6-\beta}).\]
Now the leading order term of \eqref{eq:detMatrix} is $B_{\L_5}\L_5\left(\frac{1}{\L_2}-\frac{1}{\L_3}\right)$ for $\beta>9/2$ and $B_{\lambda_{2}}\frac{\lambda_{2}}{\lambda_{3}}-B_{\lambda_{3}}\frac{\lambda_{3}}{\lambda_{2}}$ for $\beta<9/2$. Since $\L_2 \neq \L_3, \, B_{\L_2}\L_2^2\neq B_{\L_3}\L_3^2$ we have
\begin{equation*}
\text{det}M_{\lambda}(k,\eps,\omega)=    
    \begin{cases}
    O(\eps^{3-\frac{2\beta}{3}}) \qquad & \beta>\frac{9}{2} \\
O(1) \qquad & \beta<\frac{9}{2}
    \end{cases}.
\end{equation*}
For the amplitudes $a_j$ we have
\begin{equation*}
a_{2}=O(\eps^{-\frac{\beta}{3}}),\quad  a_{3}=O(\eps^{-\frac{\beta}{3}}), \quad a_{5} = \begin{cases}
O(\eps^{-\frac{\beta}{3}}) \qquad & \beta<\frac{9}{2} \\
O(\eps^{\frac{\beta}{3}-3}) \qquad & \beta>\frac{9}{2}
\end{cases}.    
\end{equation*}

$\bullet$ \noindent\underline{Case 2} $(\nu\sim\eps^6,\kappa\sim\eps^\beta, \beta>6)$:
since the leading order equation satisfied by $\lambda_5$ is $i\omega-\kappa\L_{5}^{2}= 0$, which is precisely one of the terms of the denominator of $B_{\L_5}$ (see \eqref{eq:eigen-BLs}), then we need to extract information on the next order term of the expansion of $\L_5$. We obtain
\[
\L_5=\tilde{\L}_5 \eps^{-\beta/2}+O(\eps^{\beta/2-6}),
\]
so that
\[ B_{\lambda_{5}}=\frac{O(1)+O(\eps^{\beta/2})}{O(\eps^{\beta})+O(\eps^{2\beta-12})+O(\eps^{\beta-6})}=O(\eps^{6-\beta}).\]
Now the leading order term of \eqref{eq:detMatrix} is $B_{\L_5}\L_5\left(\frac{1}{\L_2}-\frac{1}{\L_3}\right)$. Since $\L_2 \neq \L_3$ for all $\eps>0$, we have
\begin{equation*}
\text{det}M_{\lambda}(k,\eps,\omega)=O(\eps^{8-3\beta/2}).
\end{equation*}
For the amplitudes, we have
\begin{equation*}
a_{2}=O(\eps^{-2}),\quad  a_{3}=O(\eps^{-2}), \quad a_{5}=O(\eps^{\frac{3\beta}{2}-10}).
\end{equation*}

$\bullet$ \noindent\underline{Case 3} $(\nu\sim\eps^\beta,\kappa\sim\eps^6, \beta<6)$:
similarly to \underline{Case 2}, we have to push the expansion of the root $\L_5$ to the next order. This yields
\[
\L_5=\tilde{\L}_5 \eps^{-3}+O(\eps^{3-\beta}),
\]
then
\[ B_{\lambda_{5}}=\frac{O(1)+O(\eps^{3})}{O(\eps^{6-\beta})+O(\eps^{12-2\beta})+O(\eps^{6})}=O(\eps^{\beta-6}).\]
The leading order term of \eqref{eq:detMatrix} is $B_{\L_5}\L_5\left(\frac{1}{\L_2}-\frac{1}{\L_3}\right)$. Since $\L_2 \neq \L_3$, we have
\begin{equation*}
\text{det}M_{\lambda}(k,\eps,\omega)=O(\eps^{4\beta/3-9}).
\end{equation*}
For the amplitudes, we have
\begin{equation*}
a_{2}=O(\eps^{-\frac{\beta}{3}}), \quad a_{3}=O(\eps^{-\frac{\beta}{3}}), \quad a_{5}=O(\eps^{9-\frac{5\beta}{3}}).
\end{equation*}
$\bullet$ \noindent\underline{Case 4} $(\nu\sim\eps^6,\kappa\sim\eps^\beta, \beta<6)$:
recalling that $\L_5=O(\eps^{-\beta/2})$, we have
\[ B_{\lambda_{5}}=\frac{O(1)+O(\eps^{\beta/2})}{O(1)+O(\eps^{6})+O(\eps^{6-\beta})}=O(\eps^{\beta-6}).\]
This case is similar to Case 1, indeed the leading order terms of \eqref{eq:detMatrix} is $B_{\L_5}\L_5\left(\frac{1}{\L_2}-\frac{1}{\L_3}\right)$ for $\beta<8$ and $B_{\lambda_{2}}\frac{\lambda_{2}}{\lambda_{3}}-B_{\lambda_{3}}\frac{\lambda_{3}}{\lambda_{2}}$ for $\beta>8$. Since $\L_2 \neq \L_3, \, B_{\L_2}\L_2^2\neq B_{\L_3}\L_3^2,$ we have
\begin{equation*}
\text{det}M_{\lambda}(k,\eps,\omega)=    
    \begin{cases}
    O(\eps^{\beta/2-4}) \qquad & \beta<8 \\
O(1) \qquad & \beta>8
    \end{cases}.
\end{equation*}
For the amplitudes, we have
\begin{equation*}
a_{2}= O(\eps^{-2}), \quad a_{3}= O(\eps^{-2}), \quad a_{5} \sim \begin{cases}
O(\eps^{2-\frac{\beta}{2}}) \qquad & \beta<8 \\
O(\eps^{-2}) \qquad & \beta>8
\end{cases}.    
\end{equation*}

$\bullet$ \noindent\underline{Case 5} $(\nu\sim\eps^\beta,\kappa\sim\eps^\beta, \beta>6)$: recalling that $\L_5=O(\eps^{-\beta/2})$, we have
\[
B_{\L_5}=\frac{O(1)}{O(1)+O(\eps^{\beta})}=O(1).
\]
The leading order term of \eqref{eq:detMatrix} is $B_{\L_5}\L_5\left(\frac{1}{\L_2}-\frac{1}{\L_3}\right)$. Since $\L_2 \neq \L_3$, we have
\begin{equation*}
\text{det}M_{\lambda}(k,\eps,\omega)=O(\eps^{2-\beta/2}).
\end{equation*}
For the amplitudes, we have
\begin{equation*}
a_{2}=O(\eps^{-2}),\quad  a_{3}=O(\eps^{-2}), \quad a_{5}=O(\eps^{-1}).
\end{equation*}
We constructed our boundary layer wave packet $\Wbl^0$ in all the five cases, so that we have our approximate solution $\W^{\rm{app}}=\W^0_{\rm{inc}}+\Wbl^0$ to \eqref{eq:system-sec2}, exactly satisfying the boundary conditions \eqref{eq:cond-BL-sec 2}.
It remains to prove consistency of the approximate solution and to provide the size of $R_{\rm{app}}$. We have
\begin{align}\label{errore}
R_{\rm{app}}= \begin{pmatrix}
R_u^{\rm{app}}\\
R_w^{\rm{app}}\\
R_b^{\rm{app}}
\end{pmatrix}
:= -\begin{pmatrix}
   \nu\Delta u^{0}_{\text{inc}}\\\nu\Delta w^{0}_{\text{inc}}\\\kappa\Delta b^{0}_{\text{inc}}
   \end{pmatrix}
= \int_{\mathbb{R}^2} \hat{A}(k,m)  \begin{pmatrix}
   \nu (k^2+m^2) \\
   \nu (k^2+m^2)\\
   \kappa (k^2+m^2) 
   \end{pmatrix} \times X_{k,m}
   e^{ikx+imy-i\omega t} \, dk \, dm,
\end{align}
and its size (i.e $L^{2}$ norm) is $O(\nu+\kappa)$ as $\|\W^0_{\rm{inc}}\|_{L^2(\R_+^2)}=O(1)$ proved in [Lemma 2.7, \cite{BDSR19}].
\end{proof}

\begin{remark}
It is interesting to notice that $B_{\L_5}$ has always the same order of the Prandtl number $Pr=\nu/\kappa$. Moreover, there seems to be a relation between $Pr$ and the amplitude $a_5$: the smaller $Pr\ll1$ (i.e $\nu\ll\kappa$) is and the smaller its corresponding boundary layer wave packet is as well.
\end{remark}

%%%%%

\section{Size of the approximate solution}
We want to ensure that we are working with solutions with finite $L^2(\RR_+^2),L^\infty(\RR_+^2)$ norm.
First, we recall once again that exactly as in [Lemma 2.7, \cite{BDSR19}], one has
\begin{align}\label{eq:norminc}
\left \|  \W^0_{\rm{inc}} \right\|_{L^{2}}= \left \|  \W^0_{\rm{inc}} \right\|_{L^{\infty}} = O(1).
\end{align}
In the next lemma we provide a systematic recipe to automatically determine the $L^2$ and $L^\infty$ norms of the boundary layer part in
terms of the orders of the boundary layer decay ($\eps^{\beta_j}$) and of their amplitude ($\eps^{\alpha_j}$).

\begin{lemma} \label{lem:size}
Let $\W_{\rm{BL}, \eps^{\beta_j}}^{\alpha_j}$ be a boundary layer wave packet as in Definition \ref{def:BL-general}.
If, for all $(k,m)\in \text{supp} \hat{A}$, there exists a universal constant $C>0$ such that
\begin{equation}
    |a_{j}|\leq C \varepsilon^{\alpha_{j}}, \qquad \|X_{\lambda_{_{j}}}\|\leq C, \qquad \Re(\lambda_{j})\geq \frac{C}{\varepsilon^{\beta_{j}}},
\label{hp:lemma}
\end{equation}
then there exists a universal constant $K>0$ such that for the boundary layer wave packet the following inequality hold true:
\begin{align*}
 \|  \W^0_{\rm{BL},\eps^{\beta_j}} \|_{L^{\infty}(\mathbb{R}_{+}^{2})} &  \leq K\varepsilon^{\alpha_{j}+2}, \quad
  \| \W^0_{\rm{BL},\eps^{\beta_j}} \|_{L^{2}(\mathbb{R}_{+}^{2})}   \leq K\varepsilon^{\frac{2+\beta_{j}+2\alpha_{j}}{2}},
\end{align*}
for all $t \in \RR_{+}$.
\end{lemma}
\begin{proof}
As $\|\hat{A}\|_{L^{1}(\mathbb{R}^{2})} \leq C \eps^{2}$, for the $L^{\infty}$ norm it is enough to observe that for all $t \in \RR_{+}$
\[
 \|  \W^0_{\rm{BL},\eps^{\beta_j}} \|_{L^{\infty}(\mathbb{R}_{+}^{2})}  \leq c|a_{j}|\|\hat{A}\|_{L^{1}(\mathbb{R}^{2})}\leq K\varepsilon^{\alpha_{j}+2}.
\]
The $L^2$ norm is bit trickier. First of all, using Fubini's theorem, we rewrite
\[
\W^0_{\rm{BL},\eps^{\beta_j}}=-\int_{\mathbb{R}}e^{ikx} \left(\int_{\mathbb{R}}a_{j}\hat{A}(k,m)X_{\lambda_{j}}e^{-i\omega t-\lambda_{j}y}dm\right)dk.
\]
Then, by Plancherel's theorem, we have
\[
\|\W^0_{\rm{BL},\eps^{\beta_j}}\|_{L^{2}_{x}}^{2}=\int_{\mathbb{R}}\left|\int_{\mathbb{R}}a_{j}\hat{A}(k,m)X_{\lambda_{j}}e^{-i\omega t-\lambda_{j}y}dm\right|^{2}dk.
\]
Now, noticing that $\| \chi\|^{2}_{L^{1}}=O(\varepsilon^{4}), \; \| \chi^{2}\|_{L^{1}}=O(\varepsilon^{2})$ and using the hypotheses, we obtain the following inequalities 
\[
\|\W^0_{\rm{BL},\eps^{\beta_j}}\|_{L^{2}_{x}}^{2}\leq C \varepsilon^{2\alpha_{j}-4}\int_{\mathbb{R}}\left|\int_{\mathbb{R}}\sum_{\pm} \chi\left(\frac{k\pm k_{0}}{\varepsilon^{2}}\right) \chi\left(\frac{m\pm m_{0}}{\varepsilon^{2}}\right)\text{exp}(-\frac{cy}{\varepsilon^{\beta_{j}}})dm\right|^{2}dk
\]
\[
\leq C \varepsilon^{2\alpha_{j}}\text{exp}(-\frac{cy}{\varepsilon^{\beta_{j}}})\int_{\mathbb{R}}\sum_{\pm} \chi^{2}\left(\frac{k\pm k_{0}}{\varepsilon^{2}}\right)dk \leq C \varepsilon^{2\alpha_{j}+2}\text{exp}(-\frac{cy}{\varepsilon^{\beta_{j}}}).\]
Integrating the latter in y, the proof is concluded.
\end{proof}

We can now directly apply the above lemma to estimate the sizes of the boundary layer wave packets of Proposition \ref{prop:BL}, in all the different cases.

\begin{proposition}\label{prop:sizes}
Consider the approximate solution $\W^{\rm{app}}=\W^0_{\rm{inc}}+\Wbl^0$ to system \eqref{eq:system-sec2} with boundary conditions \eqref{eq:cond-BL-sec 2}, provided by Proposition \ref{prop:BL}. Then the $L^2$ and $L^\infty$ sizes of $\W^{\rm{app}}$ in the five different regimes are given by the following table.
\begin{table}[ H]
\begin{tabular}{|p{2cm} |p{1.6cm} |p{1.6cm} |p{1.6cm} |p{1.6cm}|p{1.6cm}|}
 \hline
 & Case 1 & Case 2 & Case 3 & Case 4 & Case 5\\ [2ex] \hline
$\|\W^{\rm{app}}\|_{L^\infty}$ &   $O(\eps^{2-\beta/3})$     &    $O(1)$    &    $O(\eps^{2-\beta/3})$    &    $O(1)$  & $O(1)$    \\ [2ex] \hline
$\|\W^{\rm{app}}\|_{L^{2}}$ &    $O(1)$     &     $O(1)$     &     $O(1)$     & $O(1)$ & $O(1)$  \\ [2ex]
\hline
\end{tabular}
\end{table}
More in detail, the sizes of the boundary layer wave packets in the five different regimes are summarized below.
\begin{center}
\begin{table}[H] \label{table:size-packet}
\begin{tabular}{|cc|cc|cc|cc|cc|}
\hline
\multicolumn{2}{|c|}{\multirow{2}{*}{}} &
  \multicolumn{2}{c|}{Case 1} &
  \multicolumn{2}{c|}{Case 2} &
  \multicolumn{2}{c|}{Case 3} &
  \multicolumn{2}{c|}{Case 4} \\ [2ex]\cline{3-10} 
\multicolumn{2}{|c|}{} &
  \multicolumn{1}{c|}{$\beta<\frac{9}{2}$} &
  $\beta>\frac{9}{2}$ &
  \multicolumn{2}{c|}{$\beta>6$} &
  \multicolumn{2}{c|}{$\beta<6$} &
  \multicolumn{1}{c|}{$\beta<8$} &
  $\beta>8$\\ [2ex]\hline
\multicolumn{1}{|c|}{\multirow{2}{*}{$\|\W^0_{\rm{BL},\eps^{\beta_j}}\|_{L^{\infty}}$}} & j=2,3 & \multicolumn{2}{c|}{$O(\eps^{2-\beta/3})$}    & \multicolumn{2}{c|}{$O(1)$} & \multicolumn{2}{c|}{$O(\eps^{2-\frac \beta 3})$} & \multicolumn{2}{c|}{$O(1)$}\\ [2ex]\cline{2-10} 
\multicolumn{1}{|c|}{}                  & j=5   & \multicolumn{1}{c|}{$O(\eps^{2-\frac \beta 3})$} & $O(\eps^{\frac \beta 3-1})$ & \multicolumn{2}{c|}{O($\eps^{\frac \beta 2-2}$)} & \multicolumn{2}{c|}{{O($\eps^{11-5\frac \beta 3}$)}} & \multicolumn{1}{c|}{O($\eps^{4-\frac \beta 2}$)} & $O(1)$ \\ [2ex]\hline
\multicolumn{1}{|c|}{\multirow{2}{*}{$\|\W^0_{\rm{BL},\eps^{\beta_j}}\|_{L^{2}}$}} & j=2,3 & \multicolumn{2}{c|}{$O(\eps^{1-\frac \beta 6})$}    & \multicolumn{2}{c|}{$O(1)$} & \multicolumn{2}{l|}{$O(\eps^{1-\frac \beta 6})$} & \multicolumn{2}{c|}{$O(1)$} \\ [2ex]\cline{2-10} 
\multicolumn{1}{|c|}{}                  & j=5   & \multicolumn{1}{c|}{$O(\eps^{\frac 5 2-\frac \beta 3})$} & $O(\eps^{\frac \beta 3-\frac 1 2})$ & \multicolumn{2}{c|}{O($\eps^{3\frac \beta 4-3}$)} & \multicolumn{2}{c|}{$O(\eps^{\frac{23}{2}-5\frac \beta 3})$} & \multicolumn{1}{c|}{$O(\eps^{3-\frac \beta 4})$} & $O(\eps^{\frac \beta 4-1})$ \\ [2ex]\hline
\end{tabular}
\end{table}
\end{center}
\end{proposition}
\begin{proof}
The proof follows by applying Lemma \ref{lem:size}. Actually, in \underline{Case 2} and in \underline{Case 3}, the eigenvector $X_{\L_5}$ does not satisfy all the assumptions of Lemma \ref{hp:lemma}. In this two cases, $|X_{\L_5}|$ has order  $O(\eps^{6-\beta})$ and $O(\eps^{\beta-6})$ respectively. However, one can follow the simple proof of the above lemma line by line, including the (unbounded) $|X_{\L_5}|$, to give an estimate of the $L^2$ and $L^\infty$ norms of $\W^0_{\L_5}$ in these cases as well.
\end{proof}

\begin{remark}[On the degenerate Case 5]
In \underline{Case 5}, the boundary layer wave packet associated to the root with small decay (i.e $\L_2$) does not have a finite $L^2$ norm. We have to discard this part of the solution and as a consequence, in Case 5 we cannot lift three boundary conditions. Therefore the solutions of Case 5 is only consistent, while we cannot prove stability (cfr. the statement of Theorem \ref{thm:main1}). In Case 5 ($\nu \sim \kappa \sim \eps^\beta, \beta>6$), where we only have two boundary layers $\W^0_{\rm{BL}, \eps^{\beta/2-1}}, \W^0_{\rm{BL}, \eps^{\beta/2}}$ lifting two (out of three) boundary conditions, the situation is the following:
\begin{table}[H]
\begin{tabular}{|c|c|c|}
\hline
 & $\W^0_{\rm{BL}, \eps^{\beta/2-1}}$ & $\W^0_{\rm{BL}, \eps^{\beta/2}}$ \\[2ex]\hline
$\|\cdot \|_{L^\infty}$ & $O(1)$ & $O(\eps)$ \\ [2ex]\hline
$\|\cdot \|_{L^2}$ & $O(\eps^{(\beta-6)/4})$ & $O(\eps^{(\beta/4})$ \\ [2ex]\hline
\end{tabular}
\end{table}
\end{remark}

\begin{remark}[On the $L^2$ norm of boundary layers]
Notice from the last two summarizing tables that in Case 2, 4, 5, where $\beta>6$ is allowed to be very big (and the related dissipation $\eps^\beta$ very weak), the $L^2$ norm of the boundary layers is respectively $O(\eps^{3\frac \beta 4-3}), O(\eps^{\frac \beta 4-1})$ (if $\beta>8$) and $O(\eps^\frac \beta 4)$). The $L^2$ sizes of these boundary layer wave packets is therefore decreasing as $\beta$ grows. At a first glance, this behavior could seem counterintuitive as one expects a decreasing dissipation to have a bad impact on the size of the approximate solution. However, it can be verified that the larger $\beta$ is and the larger the $H^1$ (and $H^s, s>0$) norm of these boundary layers are as well: in fact, since we are working with boundary layers close to $y=0$, the degeneracy due to the weaker dissipation is highlighted by derivatives in $y$, as the boundary layers' steepness is worse when $\beta$ is big (and the related dissipation is small).
\end{remark}

\begin{comment}
\begin{remark} (An exact solution)
We notice that there exists a obvious method to construct an exact solution with $R^{\rm{app}}=0$. Indeed we simply need to write the incident wave packet as 
$$\mathcal{U}^{inc}_{\kappa, \nu}:=X_{k,m}^{\kappa, \nu} e^{-i\omega t + i kx -  \lambda y}, $$
where now $\L=-im+\L_1$ is an accurate expression of the root of order O(1) of $\mathcal{P}(\L_1)$ of \eqref{eq:pol}.
\end{remark}
\end{comment}

\subsection*{Acknowledgment}
RB is partially supported by the GNAMPA group of the INdAM. Part of this work was developed during the master thesis of GO.


\begin{thebibliography}{99}

\bibitem{BDSR19}
{R. Bianchini, A.L. Dalibard, L. Saint-Raymond,}
\textit{Near-critical reflection of internal waves}, 
Analysis \& PDE \textbf{14} (1)   (2021), 205--249.


\bibitem{Chemin2006}
{J.-Y. Chemin, I. Gallagher, E. Grenier,}
\textit{Mathematical geophysics. An introduction to rotating fluids and the Navier-Stokes equations},
 Oxford Lecture Series in Mathematics and its Applications, 32. The Clarendon Press, Oxford University Press, Oxford, 2006. xii+250 pp. ISBN: 978-0-19-857133-9; 0-19-857133-X MR2228849.


\bibitem{danchin}
R. Danchin, L. He, \textit{The Oberbeck-Boussinesq approximation in critical spaces},  Asymptot. Anal. \textbf{84} (2013), no. 1-2, 61--102. MR3134744   

\bibitem{DY1999}
{T. Dauxois, W. R. Young, }
\textit{Near-critical reflection of internal waves}, 
J. Fluid Mech. \textbf{390}  (1999), 271--295.

\bibitem{thierry}
D. G\'erard-Varet and T. Paul, \textit{Remarks on boundary layer expansion}, Comm. Partial Differential
Equations \textbf{33}:1-3 (2008), 97–130.

\bibitem{desj}
B. Desjardins, E. Grenier, P.-L. Lions, N. Masmoudi, \textit{Incompressible limit for solutions of the isentropic Navier-Stokes equations with Dirichlet boundary conditions}. J. Math. Pures Appl. (9) 78 (1999), no. 5, 461--471. MR1697038

\bibitem{lannes}
B. Desjardins, D. Lannes, J--C. Saut, \textit{Normal mode decomposition and dispersive and nonlinear mixing in stratified fluids}, Water Waves 3 (2021), no. \textbf{1}, 153--192. MR4246392.


\bibitem{Gallay}
T. Gallay, \textit{Stability of vortices in ideal fluids : the legacy of Kelvin and Rayleigh}. Hyperbolic problems: theory, numerics, applications, 42--59, AIMS Ser. Appl. Math., 10, Am. Inst. Math. Sci. (AIMS), Springfield, MO, (2020). 

\bibitem{kakaota}
T. Kataoka, T. R. Akylas,  \textit{Viscous reflection of internal waves from a slope}, Phys. Rev.
Fluids (2020) \textbf{5}(1): 014803.

\bibitem{long1965}
R.R. Long, \textit{On the Boussinesq approximation and its role in the theory of internal waves}, Tellus \textbf{17}
(1965), pp. 46–52.

\bibitem{Maas}
L.R. Maas, D. Benielli, J. Sommeria, F.P.A. Lamb, \textit{Observation of an internal wave attractor
in a confined, stably stratified fluid}, Nature (1997) \textbf{388} (6642): 557–561.


\bibitem{Metivier}
G. M\'etivier, \textit{Small viscosity and boundary layer methods. Theory, stability analysis, and applications}. Modeling and Simulation in Science, Engineering and Technology. Birkhäuser Boston, Inc., Boston, MA, 2004. xxii+194 pp. ISBN: 0-8176-3390-1 MR2151414 

\bibitem{Rieutord}
M. Rieutord,
\textit{Fluid Dynamics: An Introduction}, Graduate Texts in Physics, Springer International Publishing
(2015), XVI+508.


\bibitem{Dauxois-new} 
D. Varma, M. Mathur, T. Dauxois, \textit{Instabilities in internal gravity waves} (2021), Mathematics in Engineering, to appear.

\end{thebibliography}
\end{document}